\documentclass[a4paper,11pt,reqno]{amsart}
\usepackage{amsmath, amssymb,amsthm}
\usepackage{enumerate}
\usepackage{cite}
\usepackage{units}
\numberwithin{equation}{section}
\newtheorem{theorem}{Theorem}[section]
\newtheorem{lemma}{Lemma}[section]

\newtheorem{corollary}{Corollary}[section]
\newtheorem{remark}{Remark}[section]
\newtheorem{definition}{Definition}[section]

\theoremstyle{definition}

\begin{document}
\bibliographystyle{amsplain}
\title{{{
Extremal mild solutions of Hilfer fractional Impulsive systems}}}
\author{Divya Raghavan*
}
\address{
Department of Mathematics,
Indian Institute of Technology Roorkee, Rooekee-247667,
Uttarakhand, India
}
\email{divyar@ma.iitr.ac.in,madhanginathan@gmail.com}
\author{
N. Sukavanam
}
\address{
Department of  Mathematics  \\
Indian Institute of Technology Roorkee,  Roorkee-247667,
Uttarakhand,  India
}
\email{n.sukavanam@ma.iitr.ac.in}
\bigskip
\begin{abstract}
The well established monotone iterative technique that is used to study the existence and uniqueness of fractional impulsive system is extended to Hilfer fractional order in this paper. The results are derived by using the method of upper and lower solution and Gronwall inequality. Also, conditions on non-compactness of measure is used effectively to prove the main result.
\end{abstract}
\maketitle
\pagestyle{myheadings}
\markboth
{Divya Raghavan and N. Sukavanam}
{Extremal mild solutions of Hilfer fractional Impulsive systems}
\section{Literature Motivation}
Over the years, the urge of finding the extremals of a function evolved in many problems, especially in, geometry, history and mechanics.
Du and Lakshmikantham \cite{Monotone-first-paper} investigated the initial value problem given as,
\begin{align*}
x'=g(t,x);\enspace x(0)=x_{0}
\end{align*}
in the Banach space $E$ with norm $\|\cdot\|$, where $x_{0}$, $x$, $g$ $\in E$ and developed a monotone iterative technique to find the existence of the extremal solutions. Ladde et al.\cite{Monotone-book} rendered a substantial theory of monotone method along with its accompanying upper and lower solutions for non-linear equations in their monograph. The basic discussion in this monograph focused on the first and second order partial differential equation. The authors of this monograph constructed two monotone sequences on the basis of quasi-monotone property. Further they claim that the converging limits $\bar{x}$ and $\underline{x}$ of the two sequences is same for the parabolic system and left an open problem regarding elliptic systems. Nieto and Cabada \cite{Monotone-second-order} studied in detail the existence of extremal solution of periodic boundary value second order system with boundary conditions,
\begin{align*}
-x''=g(t,x);\enspace x(0)=x(2\pi);\enspace x'(0)=x'(2\pi).
\end{align*}
Here the authors utilized monotone iterative procedures without the usual required boundary conditions $\alpha'(0)\geq \alpha'(2\pi)$, $\beta'(0)\leq \beta'(2\pi)$, where $\alpha$ and $\beta$ are lower and upper solutions. For both initial value problem and boundary value problem, the impulsive system finds itself a significant role in the past as well as in the present. Almost in all physical problems, the movement of the state of the system is discontinuous. Hence, widening the idea of monotone iterative technique to impulsive system was unavoidable. Liz and Nieto \cite{Monotone-second-order-impulsive} extended this approach to impulsive periodic boundary value second order system. The authors had put forward a maximum principle exclusively for impulsive functions.

Due to the advantage over integer order in many practical problems, fractional calculus is more appreciated in the past few decades. Basic theory related to fractional calculus and its applications are available in the literature in numerous research articles, books and monographs. For interested readers,\cite{Podlubny-book} by Podlubny, \cite{Kilbas-book} by Kilbas et al. and \cite{Impulsive-instant-book} by Stamova and Stamov for impulsive systems can be referred. Due to the fact that the systems with fractional order is inevitable, the study of existence and uniqueness of extremals of differential system using monotone iterative method for fractional system is vital. Lakshmikantham and Vatsala \cite{Monotone-frac-first} enhanced the monotone iterative theory to fractional order initial value problem given by
\begin{align*}
D^{\mu}(x-x(0))=f(t,x);\enspace x(0)=x_{0}
\end{align*}
where, $0<\mu<1$, $f\in C[\mathbb{R}_{0},\mathbb{R}]$. Here, $\mathbb{R}_{0}=(t,x):0\leq t\leq \alpha$ and $|x-x_{0}|\leq \beta$. Here $\alpha$ and $\beta$ are lower and upper solutions. Subsequently, many researchers focused their interest in finding the extremals using upper and lower solution method along with monotone iterative technique for both initial value and boundary value problem of fractional order differential equations. McRae \cite{Monotone-frac-Riemann} discussed the existence result using fractional monotone iterative method exclusively for differential system with Riemann-Liouville fractioanl order. Several authors including Denton and Vatsala \cite{Monotone-frac-finite} on finite fractional system, Liu et al.\cite{Monotone-frac-integro} on fractional integral systems with advanced arguments and Wang \cite{Monotone-frac-deviating} on boundary value fractional system with deviating arguments studied the existence and uniqueness using this method. Recently Agarwal et al.\cite{Monotone-frac-recent-mathematica} analyzed various cases of upper and lower solutions with initial time differences and discussed the different algorithms for distinct cases, some cases using Mittag-Leffler functions and some cases using mathematical software. The work of Mu cannot be excluded. His solitary work \cite{Monotone-frac-Mu} and the work along with Li \cite{Monotone-frac-impulsive} on monotone iterative technique to impulsive fractional system given by
\begin{align*}
D^{\alpha}x(t)+Ax(t)=&g(t,x(t)), \enspace t\in J, \enspace t\neq t_{k},\\
\Delta x|_{t=t_{k}}=&J_{k}(x(t_{k})),\enspace k=1,2,\ldots n,\\
x(0)=&x_{0},
\end{align*}
using non-compact measure and generalized Gronwall inequality is noteworthy. Here $D^{\alpha}$ is the Caputo fractional derivative of order $0<\alpha<1$, $-A$ an infinitesimal generator of an analytic semigroup $T(t)$, $t\geq 0$ and $g$ and $J_{k}$ are continuous functions. Zhang and Liang \cite{Monotone-frac-nonlocal} employed monotone iterative technique in the presence of coupled lower and upper L-quasi solution and Sadovskii's fixed point theorem.

In this regard, Gou and Li \cite{Monotone-frac-Hilfer} investigated the existence of extremal solution with the aid of lower and upper solution method for Hilfer fractional differential system. Driven by the fact that the monotone technique has not reached the impulsive Hilfer fractional differential system, this paper is projected to bridge the void. Therefore, in this paper an impulsive system with Hilfer fractional derivative considered as follows:
\begin{align}
\label{eqn:Monotone-impul-Hilfer}
\left\{
  \begin{array}{ll}
   {}_{t_{0}}D_{t}^{\mu,\nu}x(t)+Ax(t)= g(t,x(t)), \enspace t \in J, \enspace t\neq t_{k}\\
    \Delta x|_{t=t_{k}}=\phi_{k}(x(t_{k})),\enspace k=1,2,\ldots l\\
    I_{t_{0+}}^{(1-\lambda)}[x(t)]_{t=0}= x_{0}.
      \end{array}
\right.
\end{align}
Here, $D_{0}^{\mu,\nu}$ denotes the Hilfer fractional derivative of order $0<\mu<1$, type $0\leq \nu \leq1$ and
$\lambda=\mu+\nu-\mu \nu$. $-A$ is the infinitesimal generator of an analytic-semigroup of uniformly bounded linear operators $Q(t)(t\geq 0)$ on a Banach space $E$, and for $M\geq 1$, $\sup_{t\in [0,\infty)}|Q(t)|\leq M$. If the impulse effect occurs at $t=t_{k}$, for $(k=1,2,\ldots,l)$, then $\phi_{k}: E \rightarrow E$ is the mapping of the solution before the impulse effect, $x(t_{k}^{-})$, to after the impulse effect, $x(t_{k}^{+})$. It determines the size of the jump at time $t_{k}$. In other words, the impulsive moments meet the relation $\Delta I_{t_{k}}^{1-\lambda}x(t_{k})=I_{t_{k}^{+}}^{1-\lambda}x(t_{k}^{+})-I_{t_{k}^{-}}^{1-\lambda}x(t_{k}^{-})$,
where $I_{t_{k}^{+}}^{1-\lambda}x(t_{k}^{+})$ and $I_{t_{k}^{-}}^{1-\lambda}x(t_{k}^{-})$ denotes the right and the left limit of
$I_{t_{k}}^{1-\lambda}x(t)$ at $t=t_{k}$ with  $0=t_{0}<t_{1}\ldots<t_{l}<t_{l+1}=T$. In the given impulsive system, let $J=[0,T]$ and $J'=J\backslash \{t_{1},t_{2},t_{3},\ldots, t_{l}\}$, for $T>0$ and $g$ is a continuous nonlinear operator such that $g:J\times E \rightarrow E$.

The rest of the paper is framed as follows: Section 2 gives a revisit to definitions on fractional calculus and certain necessary basic theorems. Section 3 includes the proof of the main theorem and few other results related to the existence of extremal solutions and Section 4 gives the conclusion of the paper.

\section{Essential notions}
This section covers the basic results, definitions and theorems that are essential throughout this paper.

\begin{definition}\rm{\cite{Monotone-first-paper}}
In an ordered Banach space $E$, let $N$ be a proper subset of $E$. Then $N$ is said to be a cone if for $\eta \geq 0$, $\eta N\subset N$, $N+N\subset N$, $N\cap (-N)=\{0\}$ and $N=\overline{N}$ where $\overline{N}$ denotes the closure of $N$.
\end{definition}
\begin{definition}\rm{\cite{Monotone-first-paper}}
A cone $N$ is said to be normal if there exists a real number $D>0$ such that for $0\leq y \leq z$ implies $\|\cdot\|\leq D\|z\|$. Here $D$ is independent of $y$ and $z$.
\end{definition}
For detailed definition and explanation regarding positive cone of an ordered Banach space the reader may refer to \cite{Monotone-first-paper}.
Let the space of all continuous functions from $J$ to $E$ be denoted by $C(J,E)$, where $E$ is an ordered Banach space with partial order $\leq$, norm $\|\cdot \|$ and whose positive cone $N=\{x\in E:x(t)\geq \theta\}$ is normal with normal constant $D$, where $\theta$ is the zero element of $E$. $C(J,E)$ is also an ordered Banach space with norm stated as $\|x\|_{C}=\max\|x(t)\|$. Apparently, $PC(J,E)$ is an ordered Banach Space along with the norm $\|x\|_{PC}=\sup_{t\in J}\|x(t)\|$. Also, $PC_{1-\lambda}$ is an ordered Banach space with partial order $\leq$, defined as $PC_{1-\lambda}(J,E)=\{x\in PC(J,E):(t-t_{k})^{1-\lambda}x(t)\in PC(J,E)\}$ with norm $\|x\|_{PC}=\sup_{t\in J}\|(t-t_{k})^{1-\lambda}x(t)\|$ whose positive cone $N_{PC_{1-\lambda}}=\{x\in PC_{1-\lambda}(J,E):x\geq \theta\}$ is normal with the same normal constant $D$. From the defined system (\ref{eqn:Monotone-impul-Hilfer}), $x(t)$ is continuous in each $J_{k}$, where $J_{k}=(t_{k},t_{k+1}]$, for $k=1,2,\ldots,l$ with $t_{0}=0$ and $t_{l+1}=T$.

The fractional integral of order $\mu$ and for an integrable function $g$ is given as \cite{Podlubny-book},
\begin{align*}
I^{\mu}_{t}g(t)=\dfrac{1}{\Gamma(\mu)}\int^{t}_{0}(t-s)^{\mu-1}g(s)ds, \enspace \enspace 0< \mu <1.
\end{align*}
 Here $\Gamma(\cdot)$ is the gamma function. Also, the respective fractional derivative of the two classical derivatives, Caputo and
 Riemann-Liouville of order $\mu$ are given by \cite{Podlubny-book},
 \begin{align*}
^{C}D^{\mu}_{0+}g(t)=\dfrac{1}{\Gamma(1-\mu)}\int^{t}_{0}\dfrac{g'(s)}{(t-s)^{\mu}}ds, \enspace t>0,\enspace 0< \mu <1,
\end{align*}
and
\begin{align*}
^{L}D^{\mu}_{0+}g(t)=\dfrac{1}{\Gamma(1-\mu)}\left(\dfrac{d}{dt}\right)\int^{t}_{0}\dfrac{g(s)}{(t-s)^{\mu}}ds,\enspace t>0,\enspace 0< \mu <1.
\end{align*}
The Hilfer fractional derivative of order $0< \mu <1$ and type $0\leq \nu \leq 1$ of function $g(t)$ is defined by Hilfer \cite{Hilfer-book},
\begin{align*}
D^{\mu,\nu}_{0+}g(t)=I_{0+}^{\nu(1-\mu)}DI_{0+}^{(1-\nu)(1-\mu)}
\end{align*}
where $D:=\dfrac{d}{dt}$. The existence of solution for fractional system with Hilfer fractional derivative which was established by Furati et al. in \cite{Hilfer-exist-1} and Gu and Trujillo in \cite{Hilfer-remark} unlatched the flow of research on differential system with Hilfer fractional derivative. Riemann-Liouville and Caputo can be regarded as a special case of Hilfer fractional derivative, respectively as
\begin{align*}
D_{0+}^{\mu,\nu}=
\left\{
  \begin{array}{ll}
   DI_{0+}^{1-\mu}={ }^{L}D_{0+}^{\mu},\enspace \nu=0\\
   I_{0+}^{1-\mu}D={ }^{C}D^{\mu}_{0+},\enspace \nu=1.
  \end{array}
\right.
\end{align*}
The parameter $\lambda$ satisfies $\lambda=\mu+\nu-\mu \nu, \enspace 0<\lambda\leq 1$.

For $y, z \in PC_{1-\lambda}(J,E)$,  the interval $[y,z]=\{ x\in PC_{1-\lambda}(J,E)|y\leq x \leq z\}$  is ordered in $PC_{1-\lambda}(J,E)$ for $y\leq z$  and $[y(t),z(t)]=\{w\in E\enspace|y(t)\leq w \leq z(t),\enspace t\in J\}$. Fix $C^{\mu,\nu}(J,E)=\{x\in C(J,E)|D^{\mu,\nu}x \enspace\mbox{exists and} \enspace D^{\mu,\nu}x \in C(J,E)\}$. The graph norm or the $A$-norm of the Banach space denoted by $E_{A}$,  $dom(A)$ is defined as $\|\cdot\|_{A}=\|\cdot\|_{E}+\|A(\cdot)\|_{E}$. If any $x \in PC_{1-\lambda}(J,E)\cap C^{\mu,\nu}(J',E)\cap C(J',E_{A})$ satisfies all the equalities of (\ref{eqn:Monotone-impul-Hilfer}), then such an abstract function is said to be the solution of (\ref{eqn:Monotone-impul-Hilfer}).
\begin{definition}
\label{eqn:Monotone-impul-Hilfer-upper}
If $z_{0}\in PC_{1-\lambda}(J,E)\cap C^{\mu,\nu}(J',E)\cap C(J',E_{A})$ satisfies all the inequalities of
\begin{align*}
\left\{
  \begin{array}{ll}
   {}_{t_{0}}D_{t}^{\mu,\nu}z_{0}(t)+Az_{0}(t)\geq g(t,z_{0}(t)), \enspace t \in J, \enspace t\neq t_{k}\\
    \Delta z_{0}|_{t=t_{k}}\geq\phi_{k}(x(t_{k})),\enspace k=1,2,\ldots l\\
    I_{t_{0+}}^{(1-\lambda)}[z_{0}(t)]_{t=0}\geq x_{0}.
      \end{array}
\right.
\end{align*}
then, $z_{0}$ is called the upper solution of the problem \rm{(\ref{eqn:Monotone-impul-Hilfer})}.
\end{definition}
\begin{remark}
If the all the inequalities of Definition\rm{ \ref{eqn:Monotone-impul-Hilfer-upper}} are satisfied by $y_{0}$ in the reverse order, then it is a lower solution of the problem (\ref{eqn:Monotone-impul-Hilfer}).
\end{remark}
\begin{definition}
An operator family $Q(t):E\rightarrow E$ for $t\geq 0$ is supposedly positive if, for any $u\geq N$ and $t\geq0$, the inequality $Q(t)u\geq \theta$ holds.
\end{definition}
The Kuratowski measure of non-compactness measure denoted by $\alpha(\cdot)$ is defined on a bounded set. For any $t\in J$ and $B \subset C(J,E)$, define $B (t)=\{x(t):x\in B\}$. If $B$ is bounded in $C(J,E)$, then $B$ is bounded in $E$. Also, $\alpha(B(t))\leq (B)$.

The following two Lemmas are imperative for the proof of the main theorem in the next section.
\begin{lemma}\rm{\cite{Monotone-non-compact}}
\label{lem:Monotone-cgt}
Let $B_{p}=\{x_{p}\}\subset C(J,E), (p=1,2,\ldots)$ be a bounded and countable set. Then, $\alpha(B_{p}(t))$ is Lebesgue integral on $J$.  And
\begin{align*}
\alpha\left(\Big\{\int_{J}x_{p}(t)dt|_{p=1,2,\ldots,}\Big\}\right)\leq 2\int_{J}\alpha(B_{p}(t))dt.
\end{align*}
\end{lemma}
The subsequent Lemma is with reference to the generalized Gronwall inequality for fractional differential equation.
\begin{lemma}\rm{\cite{Gronwall-Inequality}}
\label{lem:Gronwall}
Suppose $b\geq 0$, $\beta> 0$ and $a(t)$ is a nonnegative function locally integrable on $0\leq t <T$ (some $T\leq +\infty$), and suppose $x(t)$ is nonnegative and locally integrable on $0\leq t <T$ with
\begin{align*}
x(t)\leq a(t)+b\int_{0}^{t}(t-s)^{\beta-1}x(s)ds
\end{align*}
on this interval; then
\begin{align*}
x(t)\leq a(t)+\int_{0}^{t}\Big[\sum_{n=1}^{\infty}\dfrac{(b\Gamma(\beta))^{n}}{\Gamma(n\beta)}(t-s)^{n\beta-1}a(s)\Big]ds, \enspace 0\leq t <T.
\end{align*}
\end{lemma}
\begin{definition}\rm{\cite{Hilfer-impulsive-inclusion}}
\label{Hilfer-mild-solution}
A function $x \in PC_{1-\lambda}(J,E)$ is called the mild solution of system (\ref{eqn:Monotone-impul-Hilfer}), if for $ t \in J$
it satisfies the following integral equation
\begin{align}
\label{eqn:Monotone-solution-origi}
x(t)=S_{\mu,\nu}(t)x_{0}+\displaystyle \sum_{i=1}^{k}S_{\mu,\nu}(t-t_{i})\phi_{i}(x(t_{i}))+
\int_{0}^{t}(t-s)^{\mu-1}P_{\mu}(t-s)g(s,x(s))ds
\end{align}
where,
\begin{align*}
S_{\mu,\nu}(t)= &I_{0+}^{\nu(1-\mu)}P_{\mu}(t),\enspace P_{\mu}(t)=
\int_{0}^{\infty}\mu \theta \xi_{\mu}(\theta)Q(t^{\mu}\theta)d\theta,\enspace
\xi_{\mu}(\theta)= \dfrac{1}{\mu}\theta^{-1-\frac{1}{\mu}}\varpi_{\mu}(\theta^{-\frac{1}{\mu}}),\\
\varpi_{\mu}(\theta)= &\dfrac{1}{\pi} \sum_{n=1}^{\infty}(-1)^{n-1}\theta^{-n\mu-1}
\dfrac{\Gamma(n\mu+1)}{n!}\sin(n\pi\mu),\enspace \theta\in (0,\infty)
\end{align*}
and $\xi_{\mu}$ is a probability density function defined on $(0,\infty)$, that is
\begin{align*}
\xi_{\mu}(\theta)\geq 0 \enspace \mbox{and} \int^{\infty}_{0}\xi_{\mu}(\theta)d\theta=1.
\end{align*}
\end{definition}
\begin{remark}
\begin{enumerate}[\rm(i)]
\item
From{\rm\cite{Hilfer-exist-1}}, when $\nu=0$, the solution reduces to the solution of classical
Riemann-Liouville fractional derivative, that is, $S_{\mu,0}(t)=P_{\mu}(t)$.
\item
Similarly when $\nu=1$, the solution reduces to the solution of classical
Caputo fractional derivative, that is $S_{\mu,1(t)}=S_{\mu}(t)$.
\end{enumerate}
\end{remark}
\begin{lemma}\rm{\cite{Hilfer-impulsive-inclusion}}
\label{lem:Relax-bounds}
If the analytic semigroup $Q(t)(t\geq 0)$ is bounded uniformly, then the operator, $P_{\mu}(t)$ and $S_{\mu,\nu}(t)$ satisfies the following bounded and
continuity conditions.
$S_{\mu,\nu}(t)$ and $P_{\mu}(t)$ are linear bounded operators and for any $x\in E$
\begin{align*}
\|S_{\mu,\nu}(t)x\|_{E}\leq \dfrac{M t^{\lambda-1}}{\Gamma(\lambda)}\|x\|_{E}\enspace \mbox{and}\enspace \|P_{\mu}(t)x\|_{E}\leq \dfrac{M}{\Gamma(\mu)}\|x\|_{E}.
\end{align*}
\end{lemma}
\section{Main results}
To prove the main theorem of this paper, an equivalent system given below is discussed. The perturbed equivalent system is valid as the constant $C\geq 0$.
\begin{align}
\label{eqn:Monotone-alternate}
\left\{
  \begin{array}{ll}
   {}_{t_{0}}D_{t}^{\mu,\nu}x(t)+(A+CI)x(t)= g(t,x(t))+Cx(t), \enspace t \in J, \enspace t\neq t_{k}\\
    \Delta x|_{t=t_{k}}=\phi_{k}(x(t_{k})),\enspace k=1,2,\ldots l\\
    I_{t_{0+}}^{(1-\lambda)}[x(t)]_{t=0}= x_{0}.
      \end{array}
\right.
\end{align}
\begin{remark}
\begin{enumerate}[\rm(i)]
\item
With reference to {\rm{\cite{Monotone-semigroup}}}, for any $C\geq 0$, $-(A+CI)$ generates an analytic semigroup $R(t)=e^{-Ct}Q(t)$ and for $t\geq 0$, $R(t)$ is positive and $\sup_{t\in[0,\infty)}\|R(t)\|\leq M^{*}$ for $M^{*}\geq 1$.
\item
Let $S^{*}_{\mu,\nu}(t)$ and $P^{*}_{\mu}(t)$ for $t\geq 0$ be two families of operators defined by
\begin{align*}
S^{*}_{\mu,\nu}(t)= &I_{0+}^{\nu(1-\mu)}P^{*}_{\mu}(t),\enspace P^{*}_{\mu}(t)=
\int_{0}^{\infty}\mu \theta \xi_{\mu}(\theta)R(t^{\mu}\theta)d\theta.
\end{align*}
\item
The above two operators are positive for $(t\geq 0)$ and for any $x\in E$,
\begin{align*}
\|S^{*}_{\mu,\nu}(t)\|\leq \dfrac{M^{*} t^{\lambda-1}}{\Gamma(\lambda)}\enspace \mbox{and}\enspace \|P^{*}_{\mu}(t)\|\leq \dfrac{M^{*}}{\Gamma(\mu)}.
\end{align*}
\end{enumerate}
\end{remark}
\begin{definition}
A function $x\in PC_{1-\lambda}(J,E)$ is said to be a mild solution of the problem {\rm{\ref{eqn:Monotone-alternate}}} if for any $x\in PC_{1-\lambda}(J,E)$, the integral equation
\begin{align*}
x(t)=S^{*}_{\mu,\nu}(t)x_{0}+\displaystyle \sum_{i=1}^{k}S^{*}_{\mu,\nu}(t-t_{i})\phi_{i}(x(t_{i}))+
\int_{0}^{t}(t-s)^{\mu-1}P^{*}_{\mu}(t-s)\Big[g(s,x(s))+Cx(s)\Big]ds.
\end{align*}
\end{definition}
The following theorem guarantees the existence of the extremal mild solution of the impulsive system (\ref{eqn:Monotone-impul-Hilfer}).
\begin{theorem}
Let $E$ be an ordered Banach space with the positive cone $N$. Assume that $Q(t)\geq 0$ and the impulsive system {\rm{(\ref{eqn:Monotone-impul-Hilfer})}} has both lower and upper solution, given by $y_{0}$ and $z_{0}$ respectively, where $y_{0}, z_{0}\in PC_{1-\lambda}$ and $y_{0}\leq z_{0}$. By adopting the monotone iterative procedure and presuming the following conditions, the impulsive system {\rm{(\ref{eqn:Monotone-impul-Hilfer})}} has the extremal solution between $y_{0}$ and $z_{0}$.
\begin{enumerate}
\item [$A(1)$]:-
For $x\in [y_{0}(t),z_{0}(t)]$ the function $g(t,x)+Cx$ is increasing in $x$, precisely, there exists a constant $C\geq 0$ such that
\begin{align*}
g(t,x_{2})-g(t,x_{1})\geq -C(x_{2}-x_{1})
\end{align*}
and $y_{0}(t)\leq x_{1}\leq x_{2}\leq z_{0}(t)$ for any $t\in J$.
\item [$A(2)$]:-
For $x\in [y_{0}(t),z_{0}(t)]$, the impulsive function is increasing. It implies
\begin{align*}
\phi_{k}(x_{1})\leq \phi_{k}(x_{2}), \enspace k=1,2,\ldots,l.
\end{align*}
\item [$A(3)$]:-
The sequence $\{x_{p}\}\subset [y_{0}(t),z_{0}(t)]$, for $t\in J$ is either decreasing or increasing monotonic sequence, in particular, there exists a constant $L\geq 0$ such that
\begin{align*}
\alpha \Big(\{g(t,x_{p})\}\Big)\leq L \alpha \Big(\{x_{p}\}\Big),\enspace p=1,2,\ldots,.
\end{align*}
\end{enumerate}
\end{theorem}
\begin{proof}
As $C>0$, the problem \ref{eqn:Monotone-impul-Hilfer} can be presented in the form of problem (\ref{eqn:Monotone-alternate}). So the exhibit of the existence of a unique mild solution for the problem (\ref{eqn:Monotone-alternate}) is suffice. For a fixed $x_{0}$, define the operator $\mathcal{G}:[y_{0},z_{0}]\rightarrow PC_{1-\lambda}(J,E)$ by
\begin{align}
\label{eqn:Monotone-mild-alternate}
(\mathcal{G}x)(t)=
\left\{
  \begin{array}{ll}
  S^{*}_{\mu,\nu}(t)x_{0}+\int_{0}^{t}(t-s)^{\mu-1}P^{*}_{\mu}(t-s)\Big[g(s,x(s))+Cx(s)\Big]ds, \enspace t\in [0,t_{1}]\\
  S^{*}_{\mu,\nu}(t)x_{0}+\displaystyle \sum_{i=1}^{k}S^{*}_{\mu,\nu}(t-t_{i})\phi_{i}(x(t_{i}))\\
  \enspace+\int_{0}^{t}(t-s)^{\mu-1}P^{*}_{\mu}(t-s)\Big[g(s,x(s))+Cx(s)\Big]ds, \enspace t\in (t_{k},t_{k+1}],\enspace k=1,2,\ldots l.
  \end{array}
\right.
\end{align}
The map $\mathcal{G}(x)(t)$ is continuous since $g$ is continuous. By the Definition \ref{Hilfer-mild-solution}, the fixed points of the operator $\mathcal{G}$ are equivalent to the mild solution of the system given in (\ref{eqn:Monotone-solution-origi}). It means,
\begin{align}
\label{eqn:Monotone-fixed}
\mathcal{G}x(t)=x(t).
\end{align}
 Now it is to be proved that the operator $\mathcal{G}$ is an increasing monotonic operator. The following steps lead to the completion of the proof. \\
\vspace{3mm} Step1:- To show $\mathcal{G}(x_{1})\leq \mathcal{G}(x_{2})$:-\\
The condition $A(1)$, can be presented in the following ways, which can be directly used in the proof. That is $\forall t\in J^{'}$,
\begin{align}
\label{eqn:Monotone:alternate-A1}
y_{0}(t)\leq x_{1}(t)&\leq x_{2}(t)\leq z_{0}(t) .\nonumber\\
g(t,x_{1}(t))+C x_{1}(t)&\leq g(t,x_{2}(t))+Cx_{2}(t).
\end{align}
Considering the case for $t\in J_{0}^{'}$, for $J_{0}^{'}=[0,t_{1}]$:- As the operators $S^{*}_{\mu,\nu}(t)$ and $P^{*}_{\mu}(t)$ are positive operators,  when the mild solutions are compared, using (\ref{eqn:Monotone:alternate-A1}), the following inequality is obtained.
\begin{align*}
\int_{0}^{t}(t-s)&^{\mu-1}P^{*}_{\mu}(t-s)\Big[g(s,x_{1}(s))+Cx_{1}(s)\Big]ds\leq\\
&\int_{0}^{t}(t-s)^{\mu-1}P^{*}_{\mu}(t-s)\Big[g(s,x_{2}(s))+Cx_{2}(s)\Big]ds.
\end{align*}
In which case, for $\forall t\in J_{k}^{'}$, with $J_{k}^{'}=(t_{k},t_{k+1}]$, $k=1,2,\ldots l$ , applying the condition $A(2)$  yields
\begin{align*}
 S^{*}_{\mu,\nu}(t)x_{1}(0)+\displaystyle \sum_{i=1}^{k}S^{*}_{\mu,\nu}(t-t_{i})\phi_{i}(x_{1}(t_{i}))+\int_{0}^{t}(t-s)^{\mu-1}P^{*}_{\mu}(t-s)\Big[g(s,x_{1}(s))+Cx_{1}(s)\Big]ds\leq\\
 S^{*}_{\mu,\nu}(t)x_{2}(0)+\displaystyle \sum_{i=1}^{k}S^{*}_{\mu,\nu}(t-t_{i})\phi_{i}(x_{2}(t_{i}))+\int_{0}^{t}(t-s)^{\mu-1}P^{*}_{\mu}(t-s)\Big[g(s,x_{2}(s))+Cx_{2}(s)\Big]ds.
\end{align*}
Eventually, $\mathcal{G}(x_{1})\leq \mathcal{G}(x_{2}).$

\vspace{3mm} Step2:- To show $y_{0}\leq\mathcal{G}(y_{0})$ ; $\mathcal{G}(z_{0})\leq z_{0}$:-\\
For the case for which $t\in J_{0}^{'}$:-

Let $D^{\mu,\nu}z_{0}(t)+Az_{0}(t)+Cz_{0}(t)=\xi(t)$. By the Definition \ref{eqn:Monotone-impul-Hilfer-upper} of the upper solution, the mild solution of the system (\ref{eqn:Monotone-impul-Hilfer}) can be written as
\begin{align*}
z_{0}(t)=&S^{*}_{\mu,\nu}(t)z_{0}(0)+\int_{0}^{t}(t-s)^{\mu-1}P^{*}_{\mu}(t-s)\xi(s)ds\\
\geq & S^{*}_{\mu,\nu}(t)x_{0}+\int_{0}^{t}(t-s)^{\mu-1}P^{*}_{\mu}(t-s)\Big[g(s,z_{0}(s))+Cz_{0}(s)\Big]ds
\end{align*}
From (\ref{eqn:Monotone-mild-alternate}), it can be observed that $z_{0}(t)\geq  \mathcal{G}(z_{0}).$\\
For $t\in J_{1}^{'}$:-
\begin{align*}
z_{0}(t)=&S^{*}_{\mu,\nu}(t)z_{0}(0)+S^{*}_{\mu,\nu}(t-t_{1})\phi_{1}(z_{0}(t_{1}))+\int_{0}^{t}(t-s)^{\mu-1}P^{*}_{\mu}(t-s)\xi(s)ds\\
\geq & S^{*}_{\mu,\nu}(t)x_{0}+S^{*}_{\mu,\nu}(t-t_{1})\phi_{1}(z_{0}(t_{1}))+\int_{0}^{t}(t-s)^{\mu-1}P^{*}_{\mu}(t-s)\Big[g(s,z_{0}(s))+Cz_{0}(s)\Big]ds\\
z_{0}(t)\geq & \mathcal{G}(z_{0}).
\end{align*}
Progressing in the same way, for every $J_{K}^{'}$, yields, in general $z_{0}(t)\geq  \mathcal{G}(z_{0})$. In the same manner, it can be proved that $y_{0}(t)\leq \mathcal{G}(y_{0})$. Altogether, it can be deduced that
\begin{align*}
y_{0}(t)\leq \mathcal{G}(y_{0})\leq \mathcal{G}(x)\leq\mathcal{G}(z_{0})\leq z_{0}(t).
\end{align*}
Whereby the conclusion may be drawn that $\mathcal{G}:[y_{0},z_{0}]\rightarrow PC_{1-\lambda}(J,E)$  is an increasing monotonic operator. Through the iterative pattern, two sequence $\{y_{p}\}$ and $\{z_{p}\}$ can be defined as,
\begin{align}
\label{eqn:Monotone-iterative-pattern}
y_{p}=\mathcal{G} (y_{p-1}); \enspace z_{p}=\mathcal{G} z_{p-1};\enspace p=1,2,\ldots.
\end{align}
Eventually, due to the monotonicity property of $ \mathcal {G}$, an increasing sequence is derived as,
\begin{align}
\label{eqn:Monotone-inequality}
y_{0}\leq y_{1}\leq y_{2}\leq \ldots \leq y_{p}\leq \ldots \leq z_{p}\leq \ldots \leq z_{2}\leq z_{1}\leq z_{0}.
\end{align}
\vspace{3mm} Step3:- Convergence of sequences $\{y_{p}\}$ and $\{z_{p}\}$ in $J^{'}$:-

Let $B_{p}=\{y_{p}|p\in \mathbb{N}\}$ and $B_{p-1}=\{y_{p-1}|p\in \mathbb{N}\}$. The pattern (\ref{eqn:Monotone-iterative-pattern}) gives the relation $B_{p}=\mathcal {G}(B_{p-1})$ and as $B_{p-1}$ can be written as $B_{p-1}=B_{p}\cup \{y_{0}\}$ for $t\in J^{'}$, it follows that
$\alpha(B_{p-1}(t))=\alpha(B_{p}(t))$. Let $\psi (t):=\alpha(B_{p}(t))$. By proving that $\psi(t)\equiv0$ on every interval $J^{'}_{k}$, it means that $\alpha(B(t_{k}))\equiv 0$ for $k=1,2,\ldots,l$, and hence $\{y_{p}\}$ is precompact in $E$ for every $t\in J$. Ultimately, by the definition of precompact,$\{y_{p}\}$ has a converging subsequence in $E$. Thus it is necessary to prove that $\psi (t)\equiv0$.

For $t\in J_{0}^{'}$ for $J_{0}^{'}=(0,t_{1}]$:-
\begin{align*}
\psi(t)=\alpha(B_{p}(t))&=\alpha(\mathcal{G}B_{p-1}(t))\\
&=\alpha\left(\Big\{\int_{0}^{t}(t-s)^{\mu-1}P^{*}_{\mu}(t-s)\Big[g(s,y_{p-1}(s))+Cy_{p-1}(s)\Big]ds\Big\}:p=1,2,\ldots\right)
\end{align*}
By using Lemma \ref{lem:Monotone-cgt} gives
\begin{align*}
\psi(t)\leq 2 \int_{0}^{t}\alpha\left(\Big\{(t-s)^{\mu-1}P^{*}_{\mu}(t-s)\Big[g(s,y_{p-1}(s))+Cy_{p-1}(s)\Big]ds\Big\}:p=1,2,\ldots\right)
\end{align*}
Applying the presumed conditions along with Lemma\ref{lem:Relax-bounds} results in
\begin{align*}
\psi(t) &\leq \dfrac{2M^{*}}{\Gamma(\mu)}\int_{0}^{t}(t-s)^{\mu-1}\left[(L+C)\alpha (B_{p-1}(s))\right]ds\\
&=\dfrac{2M^{*}}{\Gamma(\mu)}(L+C)\int_{0}^{t}(t-s)^{\mu-1}\psi(s)ds.
\end{align*}
By Lemma \ref{lem:Gronwall}, $\psi (t)\equiv 0$ on $J_{0}^{'}$. Since this holds true for all $t\in J_{0}^{'}$, in particular it holds for $t=t_{1}$. Hence $\alpha(B_{p}(t_{1}))=\alpha(B_{p-1}(t_{1}))=\psi(t_{1})=0$. Therefore $B_{P}(t_{1}$ and $B_{p-1}(t_{1})$ are precompact and subsequently $\phi (B_{p-1}(t_{1}))=0$.

Now for $t\in J_{1}^{'}$, for $J_{1}^{'}=(t_{1},t_{2}]$:-
\begin{align*}
\psi(t)=&\alpha (B_{p}(t))=\alpha(\mathcal{G}B_{p-1}(t))\\
&=\alpha\Bigg(\Big\{S^{*}_{\mu,\nu}(t)y_{p-1}(t_{1})+S^{*}_{\mu,\nu}(t)\phi_{1}(y_{p-1}(t_{1}))\\
&+\int_{0}^{t}(t-s)^{\mu-1}P^{*}_{\mu}(t-s)\Big[g(s,y_{p-1}(s))+Cy_{p-1}(s)\Big]ds\Big\}:p=1,2,\ldots\Bigg)\\
&\leq \dfrac{M^{*}b^{1-\lambda}}{\Gamma (\lambda)}[\alpha B_{p-1}(t_{1})]+\dfrac{2M^{*}}{\Gamma(\mu)}(L+C)\int_{0}^{t}\int_{0}^{t}(t-s)^{\mu-1}\psi(s)ds\\
\psi(t)&\leq \dfrac{2M^{*}}{\Gamma(\mu)}(L+C)\int_{0}^{t}(t-s)^{\mu-1}\psi(s)ds.
\end{align*}
By Lemma \ref{lem:Gronwall} $\psi (t)\equiv 0$ on $J_{1}^{'}$. Proceeding the same process interval by interval, it can be proved that $\psi(t)\equiv 0$ on every interval $J_{k}^{'}$, $k=1,2,\ldots l$. Thus $\{y_{p}\}$ is precompact and eventually for $p=1,2,\ldots,\enspace\{y_{p}\} $ has a converging subsequence and from (\ref{eqn:Monotone-inequality}), it can be observed that $\{y_{p}\}$ is itself is a converging sequence and hence there exists $\underline{x}(t)\in E$ such that $\{y_{p}\}\rightarrow \underline{x}(t)$ as $p\rightarrow \infty$, for every $t\in J$. By the definition of the operator $\mathcal{G}$ (\ref{eqn:Monotone-mild-alternate}) and using the fact that $y_{p}=\mathcal{G}y_{p-1}$, it can be written as
\begin{align*}
y_{p}(t)=
\left\{
  \begin{array}{ll}
  S^{*}_{\mu,\nu}(t)x_{0}+\int_{0}^{t}(t-s)^{\mu-1}P^{*}_{\mu}(t-s)\Big[g(s,y_{p-1}(s))+Cy_{p-1}(s)\Big]ds, \enspace t\in [0,t_{1}]\\
  S^{*}_{\mu,\nu}(t)x_{0}+\displaystyle \sum_{i=1}^{k}S^{*}_{\mu,\nu}(t-t_{i})\phi_{i}(y_{p-1}(t_{i}))\\
  \enspace+\int_{0}^{t}(t-s)^{\mu-1}P^{*}_{\mu}(t-s)\Big[g(s,y_{p-1}(s))+Cy_{p-1}(s)\Big]ds, \enspace t\in (t_{k},t_{k+1}],\enspace k=1,2,\ldots l.
  \end{array}
\right.
\end{align*}
Using Lebesgue dominated convergence theorem, as $p\rightarrow \infty$
\begin{align*}
\underline{x}(t)=
\left\{
  \begin{array}{ll}
  S^{*}_{\mu,\nu}(t)x_{0}+\int_{0}^{t}(t-s)^{\mu-1}P^{*}_{\mu}(t-s)\Big[g(s,\underline{x}(s))+C\underline{x}(s)\Big]ds, \enspace t\in [0,t_{1}]\\
  S^{*}_{\mu,\nu}(t)x_{0}+\displaystyle \sum_{i=1}^{k}S^{*}_{\mu,\nu}(t-t_{i})\phi_{i}(\underline{x}(t_{i}))\\
  \enspace+\int_{0}^{t}(t-s)^{\mu-1}P^{*}_{\mu}(t-s)\Big[g(s,\underline{x}(s))+C\underline{x}(s)\Big]ds, \enspace t\in (t_{k},t_{k+1}],\enspace k=1,2,\ldots l.
  \end{array}
\right.
\end{align*}
It can be observed that $\underline{x}\in PC_{1-\lambda}$ and $\underline{x}=\mathcal{G}\underline{x}$. In a similar manner, it can be proved that $\exists$ $\overline{x}\in PC_{1-\lambda}$ such that $\overline{x}=\mathcal{G}\overline{x}$. With the monotonicity property of $\mathcal{G}$, it can be concluded that $y_{0}\leq \underline{x}\leq \overline{x}\leq z_{0}$. This proves that there exists minimal and maximal solutions $\underline{x}$ and $\overline{x}$ respectively in $[y_{0},z_{0}]$ for the given impulsive system (\ref{eqn:Monotone-impul-Hilfer}).
\end{proof}
\begin{remark}
 The above proved theorem holds for the case when the positive cone $N$ which is normal is replaced with positive cone which is regular. For detailed proof \rm{\cite[Corollary 3.3]{Monotone-frac-impulsive}} may be referred.
\end{remark}
\begin{corollary}
\label{Coro:Monotone-leq}
In an ordered Banach space $E$, let $N$ be the positive cone with normal constant $D$. Supposing that the operator $Q(t)$ is positive for $t\in J$. If the conditions $A(1)$ and $A(2)$ are satisfied combined with the following condition, then the condition $A(3)$ is automatically true.
\begin{enumerate}
\item [$A(4)$]:- There exists a constant $C^{*}$ such that
\begin{align*}
g(t,x_{2})-g(t,x_{1})\leq C^{*}(x_{2}-x_{1})
\end{align*}
and $y_{0}(t)\leq x_{1}\leq x_{2}\leq z_{0}(t)$ for any $t\in J$.
\end{enumerate}
\end{corollary}
\begin{proof}
Let $\{x_{p}\}$ and $\{x_{q}\}$ be two increasing sequences such that $\{x_{p}\},\{x_{q}\}\subset [y_{0}(t),z_{0}(t)]$, for $t\in J$ and $p\leq q$.
By the condition $A(1)$ and $A(4)$,
\begin{align*}
\theta \leq g(t,x_{q})-g(t,x_{p})+C(x_{q}-x_{p})\leq (C^{*}+C)(x_{q}-x_{p}).
\end{align*}
Using the normality constant of the positive cone $N$, it reduces to,
\begin{align*}
\|g(t,x_{q})-g(t,x_{p})\|\leq (DC^{*}+DC+C)\|x_{q}-x_{p}\|.
\end{align*}
Let $L=(DC^{*}+DC+C)$.  By the definition of measure of non-compactness the above equation reduces to,
$\alpha\left(\{g(t,x_{p})\}\right)\leq L\alpha(\{x_{p}\})$.
Thus the condition $A(3)$ is reduced.
\end{proof}
Now it is necessary to prove the uniqueness of the mild solution that lie in $[y_{0},z_{0}]$.
\begin{theorem}
An impulsive fractional system\rm{ (\ref{eqn:Monotone-impul-Hilfer})} is said to have an unique mild solution that lie between $[y_{0},z_{0}]$, where $y_{0}\in PC_{1-\lambda}$ and $z_{0}\in PC_{1-\lambda}$ are the lower and upper solution with $y_{0}\leq z_{0}$, if the conditions $A(1)$, $A(2)$ and the Corollary \rm{(\ref{Coro:Monotone-leq})} holds .
\end{theorem}
\begin{proof}
If $\overline{x}$ and $\underline{x}$ are the maximal and the minimal solution of the impulsive system (\ref{eqn:Monotone-impul-Hilfer}), then to prove the uniqueness, it has to be proved that $\overline{x}=\underline{x}$. Like in the previous proof, the theorem is proved interval by interval. Let $t\in J_{0}^{'}$. Using (\ref{eqn:Monotone-fixed}) for both the solutions results in,
\begin{align*}
\theta \leq &\enspace \overline{x}(t)-\underline{x}(t)=\mathcal{G}\overline{x}(t)-\mathcal{G}\underline{x}(t)\\
= & \int_{0}^{t}(t-s)^{\mu-1}P^{*}_{\mu}(t-s)\big[(g(s,\overline{x}(t))- g(s,\underline{x}(t)))+C(\overline{x}(t)-\underline{x}(t))\big]ds\\
\leq & \int_{0}^{t}(t-s)^{\mu-1}P^{*}_{\mu}(t-s)(C^{*}+C)(\overline{x}(t)-\underline{x}(t))ds.
\end{align*}
Using the normality of the positive cone $N$,
\begin{align*}
\|\overline{x}(t)-\underline{x}(t)\|\leq \dfrac{DM^{*}}{\Gamma (\mu)}(C^{*}+C)\int_{0}^{t}(t-s)^{\mu-1}\|\overline{x}(t)-\underline{x}(t)\|ds.
\end{align*}
By Gronwall inequality, $\|\overline{x}(t)-\underline{x}(t)\|=0$. Which implies $\overline{x}(t)\equiv\underline{x}(t)$. Calculating in the similar way results in $\overline{x}(t)\equiv\underline{x}(t)$ for $t\in J_{k}^{'}$, for $k=1,2,\ldots,l$, that is for every interval $J_{k}^{'}$. The uniqueness is thus proved.
\end{proof}
\section{Conclusion}
The study of existence of upper and lower solution of impulsive fractional system with Hilfer fractional derivative is not yet done so far. The objective of this paper is to study the existence of the mild solutions for an impulsive Hilfer fractional evolution equation where the operator generates positive analytic semigroup.  The existence can be further studied when the operator $A$ generates a $C_{0}$-semigroup.

\section*{Acknowledgement}
The work of the first author is supported by the Women Scientist Scheme A (WOS-A) of the Department of Science and Technology, India,
through Project No. SR/WOS-A/PM-18/2016.

\end{document}